\documentclass{amsart}
\usepackage{amssymb}
\usepackage{mathrsfs}
\usepackage{graphicx}
\usepackage{comment}
\usepackage{enumitem}

\newcommand{\be}{\begin{equation}}
\newcommand{\ee}{\end{equation}}
\newcommand{\ba}{\begin{align}}
\newcommand{\ea}{\end{align}}

\newtheorem{theorem}{Theorem}[section]
\newtheorem{lemma}[theorem]{Lemma}

\newtheorem*{proposition*}{Proposition}
\newtheorem{corollary}[theorem]{Corollary}
\newtheorem*{theorem*}{Theorem}
\newtheorem*{corollary*}{Corollary}
\newtheorem*{cor*}{Corollary}	
	
\theoremstyle{definition}

\newtheorem{remark}[theorem]{Remark}

\newtheorem{conjecture}[theorem]{Conjecture}

{\begin{list}{}{%
\settowidth{\labelwidth}{\textsf{{\it #1.}}}%
\setlength{\labelsep}{4mm}%
\setlength{\leftmargin}{\labelwidth}%
\addtolength{\leftmargin}{\labelsep}%
}}%
{\end{list}}

\def\beq{\begin{equation}}\def\enq{\end{equation}}

{\begin{list}{}{%
\settowidth{\labelwidth}{\textsf{{\it #1.}}}%
\setlength{\labelsep}{2mm}%
\setlength{\leftmargin}{\labelwidth}%
\addtolength{\leftmargin}{\labelsep}%
\addtolength{\leftmargin}{4mm}%
\setlength{\itemsep}{6pt}%
\setlength{\listparindent}{0pt}%
\setlength{\topsep}{3pt}%
}}%

\usepackage{color}
\usepackage[normalem]{ulem} 
\usepackage{soul}

\usepackage{xcolor}

\allowdisplaybreaks

\begin{document}

\title[Extending Recent Congruence Results on $(\ell,\mu)$-regular overpartitions]{Extending Recent Congruence Results on $(\ell,\mu)$-regular overpartitions}

\author[B. Paudel]{Bishnu Paudel}

\author[J. Sellers]{James A. Sellers}

\author[H. Wang]{Haiyang Wang}

\address{Mathematics and Statistics Department\\
         University of Minnesota Duluth\\
         Duluth, MN 55812, USA}
\email{bpaudel@d.umn.edu, jsellers@d.umn.edu, wan02600@d.umn.edu}
\begin{abstract}
	Recently, Alanazi, Munagi, and Saikia employed the theory of modular forms to investigate the arithmetic properties of the function $\overline{R_{\ell,\mu}}(n)$, which enumerates the overpartitions of $n$ where no part is divisible by either $\ell$ or $\mu$, for various integer pairs $(\ell, \mu)$. In this paper, we substantially extend several of their results and establish infinitely many families of new congruences. Our proofs are entirely elementary, relying solely on classical $q$-series manipulations and dissection formulas.		
\end{abstract}			
		
\maketitle
\section{Introduction}
A \textit{partition} of a non-negative integer $n$ is a sequence of positive integers $\lambda_1 \geq \lambda_2 \geq \cdots \geq \lambda_r$ such that $\sum_{i=1}^r \lambda_i = n$. Each integer $\lambda_i$ in this sequence is referred to as a \textit{part} of the partition.

An \textit{overpartition} of $n$ extends the notion of ordinary partitions by allowing the first occurrence of each
part size to be overlined. To illustrate, the fourteen overpartitions of $4$ are:
\[(4), (\overline{4}), (3,1), (\overline{3},1), (3,\overline{1}), (\overline{3},\overline{1}), (2,2), (\overline{2},2),\]
\[(2,1,1), (\overline{2},1,1), (2,\overline{1},1), (\overline{2},\overline{1},1), (1,1,1,1), (\overline{1},1,1,1).\]
The number of overpartitions of $n$ is denoted by $\overline{p}(n)$, with $\overline{p}(0) := 1$. The study of overpartitions dates back to MacMahon~\cite{MacMahon}, and was later revisited by Corteel and Lovejoy~\cite{Corteel-Lovejoy}, sparking renewed interest in the subject. To describe the generating function for overpartitions, recall the $q$-Pochhammer symbol
\[(a; q)_{\infty} := \prod_{i=0}^{\infty} (1 - a q^i),\] 
and use the shorthand notation 
\[f_k := (q^k; q^k)_{\infty}.\] 
Corteel and Lovejoy~\cite{Corteel-Lovejoy} noted that the generating function for $\overline{p}(n)$ is given by
\begin{equation*}
	\sum_{n = 0}^{\infty} \overline{p}(n)q^{n} = \prod_{n = 1}^{\infty} \frac{1 + q^{n}}{1 - q^{n}} = \frac{f_{2}}{f_{1}^{2}}.
\end{equation*}
For the arithmetic properties of $\overline{p}(n)$, we refer the reader to~\cite{Chen-Sun-Wang-Zhang, Fortin-Jacob-Mathieu, Hirschhorn-Sellers_2005, Mahlburg, Treneer, Xia}.

An overpartition is called \textit{$\ell$-regular} if none of its parts is divisible by $\ell$. The number of $\ell$-regular overpartitions of $n$, denoted by $\overline{A}_{\ell}(n)$, has the generating function
\begin{equation*}
	\sum_{n = 0}^{\infty} \overline{A}_{\ell}(n)q^{n} = \frac{f_{\ell}^{2}f_{2}}{f_{1}^{2} f_{2\ell}}.
\end{equation*}
Lovejoy~\cite{Lovejoy} studied this family of functions in detail, 
and the arithmetic properties of $\overline{A}_{\ell}(n)$ have been explored by several authors; see~\cite{Alanazi-Munagi-Sellers, Barman-Ray, Ray-Barman, Shen}.

A natural extension arises by considering overpartitions that are simultaneously $\ell$-regular and $\mu$-regular for coprime integers $\ell, \mu > 1$. Such overpartitions, called \textit{$(\ell, \mu)$-regular}, contain no part divisible by $\ell$ or $\mu$. The number of $(\ell, \mu)$-regular overpartitions of $n$ is denoted by $\overline{R_{\ell,\mu}}(n)$ with generating function
\begin{equation*}
	\sum_{n=0}^{\infty} \overline{R_{\ell,\mu}}(n) q^n = \frac{f_2 f_{\ell}^2 f_{\mu}^2 f_{2 \mu \ell}}{f_1^2 f_{2 \ell} f_{2 \mu} f_{\mu \ell}^2}.
\end{equation*}

In their work~\cite{Nadji-Ahmia-Ramirez}, Nadji, Ahmia, and Ramírez investigated the arithmetic properties of $\overline{R_{\ell,\mu}}(n)$ for the pairs
\[
(\ell,\mu) \in \{(4,3), (4,9), (8,3), (8,9)\},
\] 
employing dissection formulas and Smoot's implementation of Radu's algorithm. Alanazi, Munagi, and Saikia~\cite{Alanazi-Munagi-Saikia} established numerous congruences for the following combinations:
\[
(\ell,\mu) \in \{(2,3), (2,5), (3,5), (4,3), (4,9), (8,27), (16,81)\}.
\]
Their approach utilized modular forms and Radu's algorithm. In the concluding remarks of~\cite{Alanazi-Munagi-Saikia}, the authors suggested that it is desirable to find elementary proofs for their congruences. Recently, Ghoshal and Jana~\cite{Ghoshal-Jana} provided such proofs for several of these congruences and also introduced additional divisibility relations.

This article contributes further to the topic by extending various congruences from~\cite{Alanazi-Munagi-Saikia} using elementary $q$-series identities. We also supply elementary proofs for several results from~\cite{Alanazi-Munagi-Saikia} that were not 
addressed in~\cite{Ghoshal-Jana} and \cite{Nadji-Ahmia-Ramirez}.
		
We first analyze the pair $(\ell,\mu)=(2,3)$ and deduce the following internal congruence modulo $3$.
\begin{theorem}\label{T23}
    For all $n\geq0$, we have 
    \begin{align*}
        \overline{R_{2,3}}(27n)&\equiv  \overline{R_{2,3}}(3n) \pmod{3}.
        \end{align*}
\end{theorem}

In the following corollary, we establish infinitely many families of congruences modulo $6$.
\begin{corollary}\label{C23}
    For all $\beta\geq1$ and $n\geq0$, we have 
\begin{equation}\label{2306}
\overline{R_{2,3}}\left(3^\beta(3n+2)\right)\equiv0\pmod{6}.
    \end{equation}
\end{corollary}
The congruence $(10)$ given in \cite[Theorem 4.1]{Alanazi-Munagi-Saikia},
$$\overline{R_{2,3}}(9n+6)\equiv0\pmod{6},$$
is a special case of \eqref{2306} with $\beta=1$, and an elementary proof was given in \cite[Theorem 3.1]{Ghoshal-Jana}.

Next, we consider the pair $(\ell,\mu)=(4,3)$
and prove the following new congruences.
\begin{theorem}\label{T432n} For all $n\geq1$, we have
\begin{equation}
\label{432n04}\overline{R_{4,3}}(2n)\equiv0\pmod{4}.
\end{equation}
If $n\geq1$ is not an odd square, then we have
\begin{equation}
\label{432n08}\overline{R_{4,3}}(2n)\equiv0\pmod{8}.
\end{equation}
\end{theorem}
We note that Theorem \ref{T432n} implies the congruences $(16)$ and $(17)$ given in \cite[Theorem 3.1]{Nadji-Ahmia-Ramirez}:
\begin{align*}
    \overline{R_{4,3}}(4n+2)&\equiv0\pmod{4},\\
    \overline{R_{4,3}}(8n+4)&\equiv0\pmod{8}.
\end{align*}

Finally, we investigate the pair $(\ell,\mu)=(4,9)$. An elementary proof of $\overline{R_{4,9}}(4n)\equiv 0\pmod{3}$ was given in \cite[Theorem 4.1]{Nadji-Ahmia-Ramirez}; however, we extend this congruence to 
the modulus $12$ and provide a different elementary proof here.
\begin{theorem}\label{T494n012} For $n\geq1$, we have
    \begin{equation}
       \label{494n012} \overline{R_{4,9}}(4n)\equiv 0\pmod{12}.
    \end{equation}
\end{theorem}

Theorem \ref{T494n012} implies the congruences (16) and (17) given in \cite[Theorem 4.1]{Alanazi-Munagi-Saikia}:
\begin{align*}
    \overline{R_{4,9}}(8n+4)&\equiv0\pmod{12},\\
     \overline{R_{4,9}}(12n+4)&\equiv0\pmod{12}.
\end{align*}
 It also confirms the following conjecture proposed in \cite{Alanazi-Munagi-Saikia}.
\begin{conjecture}\cite[Conjecture 5.1]{Alanazi-Munagi-Saikia}
    For $n\geq0,\ell\geq2$, and $1\leq k\leq\ell$, we have \begin{equation*}\label{conj}
        \overline{R_{4,9}}(4\ell n+4k)\equiv 0\pmod{6}.
    \end{equation*}
\end{conjecture}

We also prove a new congruence modulo 8 satisfied by $\overline{R_{4,9}}$.
\begin{theorem}\label{T493n08} For $n\geq1$, we have
   \begin{equation*} \label{493n08}\overline{R_{4,9}}(3n)\equiv 0\pmod{8}.
   \end{equation*}
\end{theorem}
In the next theorem, we establish infinitely many families of congruences modulo $24$.
\begin{theorem}\label{T494knr}
For $k\geq2$, $n\geq0$, and $r$, $1\leq r\leq k-1$, a quadratic nonresidue modulo k, we have
\begin{equation}\label{494knr08}
\overline{R_{4,9}}\left(4(kn+r)\right)\equiv0\pmod{24}.
\end{equation}
If $n\geq1$ is not an odd square, then we have 
    \begin{equation}\label{49n04}
        \overline{R_{4,9}}(n)\equiv0\pmod{4}.
    \end{equation}
    \end{theorem}
When $k=4$ and $r=2$, it follows from \eqref{494knr08} that
\begin{equation*}
    \overline{R_{4,9}}(16n+8)\equiv0\pmod{24},
\end{equation*}
which is the congruence $(19)$ given in \cite[Theorem 4.1]{Alanazi-Munagi-Saikia}.

\bigskip 

To close our work, we consider the congruences $(20)$, $(21)$, and $(22)$ given in \cite[Theorem 4.1]{Alanazi-Munagi-Saikia} in order to provide an elementary proof for each. 
\begin{theorem}\label{T49} For $n\geq 0$, we have
    \begin{align}
     \label{4918n1296} \overline{R_{4,9}}(18n+12)&\equiv0\pmod{96},\\
     \label{4918n1548} \overline{R_{4,9}}(18n+15)&\equiv0\pmod{48},\\
     \label{4924n20216} \overline{R_{4,9}}(24n+20)&\equiv0\pmod{216}.
    \end{align}
\end{theorem}
\section{Preliminaries}
We recall the definition of Ramanujan's theta function
$$\varphi(q):=\sum_{n=-\infty}^\infty q^{n^2}=1+2\sum_{n\geq1}^\infty q^{n^2},$$
and observe that
\begin{equation}\label{phiq^418}
\varphi(q)^4\equiv1\pmod{8}.
\end{equation}

We now collect several identities that we will utilize in our work below.
\begin{lemma} We have
    \begin{align}
         \label{f25/f12f42} \varphi(q)&=\frac{f_2^5}{f_1^2f_4^2},\\
         \label{sf12/f2}\varphi(-q)&=\sum_{k=-\infty}^{\infty}(-1)^kq^{k^2}=\frac{f_1^2}{f_2},\\
        \label{phiq} \varphi(q)&=\frac{\varphi(-q^2)^2}{\varphi(-q)}\\\label{1/phi-q}\frac{1}{\varphi(-q)}&=\varphi(q)\varphi(q^2)^2\varphi(q^4)^4\cdots.
    \end{align}
\end{lemma}
\begin{proof}
Identities \eqref{f25/f12f42}, \eqref{sf12/f2}, \eqref{phiq}, and \eqref{1/phi-q} correspond to (1.5.6), (1.5.8), (1.5.14), and (1.5.16), respectively, in \cite{Hirschhorn}.
\end{proof}
\begin{lemma} We have
		\begin{align}
			\label{f13} f_1^3&=\sum_{m\geq0}(-1)^m(2m+1)q^{m(m+1)/2},\\
			\label{1/f14}\frac{1}{f_1^4}&=\frac{f_4^{14}}{f_2^{14}f_8^4}+4q\frac{f_4^2f_8^4}{f_2^{10}}\\
			\label{f9/f1}\frac{f_9}{f_1}&=\frac{f_{12}^3f_{18}}{f_2^2f_6f_{36}}+q\frac{f_4^2f_6f_{36}}{f_2^3f_{12}},\\
			\label{f33/f1} \frac{f_3^3}{f_1}&=\frac{f_4^3f_6^2}{f_2^2f_{12}}+q\frac{f_{12}^3}{f_4},\\
			\label{f3/f13}\frac{f_3}{f_1^3}&=\frac{f_4^6f_6^3}{f_2^9f_{12}^2}+3q\frac{f_4^2f_6f_{12}^2}{f_2^7},\\
			\label{f32/f12} \frac{f_3^2}{f_1^2}&=\frac{f_4^4f_6f_{12}^2}{f_2^5f_8f_{24}}+2q\frac{f_4f_6^2f_8f_{24}}{f_2^4f_{12}},\\
			\label{f1/f33}\frac{f_1}{f_3^3}&=\frac{f_2f_4^2f_{12}^2}{f_6^7}-q\frac{f_2^3f_{12}^6}{f_4^2f_6^9},\\
			\label{f3/f1}\frac{f_3}{f_1}&=\frac{f_4f_6f_{16}f_{24}^2}{f_2^2f_8f_{12}f_{48}}+q\frac{f_6f_8^2f_{48}}{f_2^2f_{16}f_{24}}\\
			\label{f1f3}f_1f_3&=\frac{f_2f_8^2f_{12}^4}{f_4^2f_6f_{24}^2}-q\frac{f_4^4f_6f_{24}^2}{f_2f_8^2f_{12}^2}.
		\end{align}
	\end{lemma}
	\begin{proof}
		Identity~\eqref{f13} is given as Equation~(1.7.1) in~\cite{Hirschhorn}, while \eqref{1/f14} appears in~\cite{Brietzke-daSilva-Sellers}. The relation~\eqref{f9/f1} is established in~\cite[Lemma~3.5]{Xia-Yao}. Identities~\eqref{f33/f1}, \eqref{f3/f13}, and \eqref{f32/f12} are all proved in~\cite{Xia-Yao_2013}. By substituting $q$ with $-q$ in~\eqref{f33/f1} and applying the fact that
		\begin{equation}
			(-q;-q)_\infty = \frac{f_2^3}{f_1 f_4},
		\end{equation}
		we deduce~\eqref{f1/f33}. Lastly, \eqref{f3/f1} and~\eqref{f1f3} correspond to (31) and (32), respectively, in~\cite{daSilva-Sellers}.
	\end{proof}
   \begin{lemma}
		\begin{align}
			\label{f1f2}f_1f_2&=\frac{f_6f_9^4}{f_3f_{18}^2}-qf_9f_{18}-2q^2\frac{f_3f_{18}^4}{f_6f_9^2},\\
			\label{f2/f12}\frac{f_2}{f_1^2}&=\frac{f_6^4f_9^6}{f_3^8f_{18}^3}+2q\frac{f_6^3f_9^3}{f_3^7}+4q^2\frac{f_6^2f_{18}^3}{f_3^6},\\
			\label{f1/f4}\frac{f_1}{f_4}&=\frac{f_6f_9f_{18}}{f_{12}^3}-q\frac{f_3f_{18}^4}{f_9^2f_{12}^3}-q^2\frac{f_6^2f_9f_{36}^3}{f_{12}^4f_{18}^2},\\
			\label{1/f1f2}\frac{1}{f_1f_2}&=\frac{f_9^9}{f_3^6f_6^2f_{18}^3}+q\frac{f_9^6}{f_3^5f_6^3}+3q^2\frac{f_9^3f_{18}^3}{f_3^4f_6^4}-2q^3\frac{f_{18}^6}{f_3^3f_6^5}+4q^4\frac{f_{18}^9}{f_3^2f_6^6f_9^3},\\
			\label{f12/f2}\frac{f_1^2}{f_2}&=\frac{f_9^2}{f_{18}}-2q\frac{f_3f_{18}^2}{f_6f_9}.
		\end{align}
	\end{lemma}
	\begin{proof}
		A proof of~\eqref{f1f2} is provided in~\cite{Hirschhorn-Sellers_2014}, while \eqref{f2/f12} is established in~\cite{Hirschhorn-Sellers_2005}. For a proof of~\eqref{f1/f4}, see~\cite[Lemma~5.1]{Das-Maity-Saikia}. Equation~\eqref{1/f1f2} is given as Lemma~9 in~\cite{Sellers}. Identity~\eqref{f12/f2} appears as~(14.3.2) in~\cite{Hirschhorn}.
	\end{proof}
\begin{lemma}
    For a prime $p$ and positive integers $k$ and $l$,
    \begin{equation}\label{modp^k}
        f_l^{p^k}\equiv f_{lp}^{p^{k-1}}\pmod{p^k}.
    \end{equation}
\end{lemma}
\begin{proof}
See~\cite[Lemma~3]{daSilva-Sellers}.
\end{proof}
Next, we quickly prove that $\overline{R_{\ell,\mu}}(n)$ is almost always even.
\begin{lemma} For $n\geq 1$, we have 
    \begin{equation} \label{Tlu02}
       \overline{R_{\ell,\mu}}(n)\equiv0\pmod{2}.
    \end{equation}
\end{lemma}
\begin{proof} Thanks to \eqref{modp^k},
    \begin{align*}
        \sum_{n\geq 0}\overline{R_{\ell,\mu}}(n)=\frac{f_2f_\ell^2f_\mu^2f_{2\mu\ell}}{f_1^2f_{2l}f_{2\mu}f_{\mu\ell}^2}\equiv 1\pmod{2}.
    \end{align*} \end{proof}

\begin{lemma}\label{Gq=qGq}
   Let $k,s,t>0$ be integers and let $G(q)$ be a $q$-series. If $G(q)\equiv q^sG(q^t)\pmod{k}$, then $G(q)\equiv 0\pmod{k}$.
\end{lemma}
\begin{proof}
    We write
    \begin{equation*}
        G(q)=\sum_{n\geq0}a_nq^n.
    \end{equation*}
We use induction to show that $a_n\equiv0\pmod{k}$ for all $n$. Since $G(q)\equiv q^sG(q^t)\pmod{k}$, we have
\begin{equation}\label{A}
\sum_{n\geq0}a_nq^n\equiv\sum_{n\geq0}a_nq^{tn+s}\pmod{k}.
    \end{equation}
Clearly, $a_0\equiv0\pmod{k}$. Suppose $a_n\equiv0\pmod k$ for $0\leq n\leq m$. Then, from \eqref{A}, we get
\begin{equation*}
\sum_{n\geq m+1}a_nq^n\equiv\sum_{n\geq m+1}a_nq^{tn+s}\pmod{k},
    \end{equation*}
from which we can conclude that $a_{m+1}\equiv0\pmod{k}$.
\end{proof}

We conclude this section by citing some relevant results on $\overline{pp_0}(n)$, the number of overpartition pairs of n into odd parts. These results will appear again in Remark \ref{remark:ppo} below. 
The generating function for this partition function is given by
\begin{equation}\label{pponvarphi}
    \sum_{n\geq0}\overline{pp_o}(n)q^n=\frac{\varphi(q)}{\varphi(-q)}.
\end{equation}
\begin{lemma}\cite[Theorem 1.2]{Adiga-Dasappa}\label{ppo2^k}
    For all $n\geq 0$ and $k\geq4$, we have 
    \begin{align*}
        \overline{pp_o}(2^{k+1}n)&\equiv0\pmod{2^{3k+2}}, \quad n\geq1,\\
        \overline{pp_o}(2^k(4n+3))&\equiv0\pmod{2^{3k+1}},\\
        \overline{pp_o}(2^k(8n+5))&\equiv0\pmod{2^{3k}}.
    \end{align*}
\end{lemma}
With the tools mentioned above, we are now ready to establish our main findings.

\section{Proof of Theorem \ref{T23} and Corollary \ref{C23}}
\subsection*{Proof of Theorem \ref{T23}}
We have 
\begin{align*}
    \sum_{n\geq0}\overline{R_{2,3}}(n)q^n&=\frac{f_2^3f_3^2f_{12}}{f_1^2f_4f_6^3}\\
    &=\frac{f_2^3f_3^2f_{12}}{f_1^3f_6^3}\left(\frac{f_1}{f_4}\right)\\
    &\equiv\frac{f_6f_3^2f_{12}}{f_3f_6^3}\left(\frac{f_1}{f_4}\right) \pmod{3} \quad (\text{thanks to} \,\ \eqref{modp^k})\\
    &=\frac{f_3f_{12}}{f_6^2}\left(\frac{f_6f_9f_{18}}{f_{12}^3}-q\frac{f_3f_{18}^4}{f_9^2f_{12}^3}-q^2\frac{f_6^2f_9f_{36}^3}{f_{12}^4f_{18}^2}\right) \quad(\text{thanks to}\,\ \eqref{f1/f4}).
\end{align*}
We now extract the terms in which the exponents of $q$ are of the form $3n$ and replace $q^3$ by $q$ to obtain
\begin{align}
   \notag \sum_{n\geq0}\overline{R_{2,3}}(3n)q^n&\equiv\frac{f_1f_{4}}{f_2^2}\left(\frac{f_2f_3f_6}{f_4^3}\right)\pmod{3}\\
    \notag&=\frac{f_1f_4f_6^2}{f_2f_3f_4^3}\left(\frac{f_3^2}{f_6}\right)\\
   \notag &\equiv \frac{f_1f_4f_2^6}{f_2f_1^3f_{4}^3}\left(\frac{f_3^2}{f_6}\right) \pmod{3} \quad (\text{thanks to} \,\ \eqref{modp^k})\\
  \notag &=\frac{f_2^5}{f_1^2f_4^2}\left(\frac{f_3^2}{f_6}\right)\\
 \label{R233n} &= \varphi(q)\frac{f_3^2}{f_6} \quad (\text{thanks to } \eqref{f25/f12f42})\\
\label{R233nA}&=\left(\varphi(q^9)+qH(q^3)\right)\left(\frac{f_{27}^2}{f_{54}}-2q^3\frac{f_9f_{54}^2}{f_{18}f_{27}}\right),
\end{align}
for some function $H$, where the last equality follows by using \eqref{f12/f2} and observing that $n^2\not\equiv2\pmod{3}$ and $3|n^2$ if and only if $9|n^2$. Thus, collecting the terms in which the exponents of $q$ are the form $9n$ gives
\begin{align*}
     \sum_{n\geq0}\overline{R_{2,3}}(27n)q^{9n}&\equiv\varphi(q^9)\frac{f_{27}^2}{f_{54}}\pmod{3}.  
\end{align*}
We replace $q^9$ by $q$ to get
\begin{equation}
   \label{R2327n}  \sum_{n\geq0}\overline{R_{2,3}}(27n)q^{n}\equiv\varphi(q)\frac{f_{3}^2}{f_{6}}\pmod{3}. 
\end{equation}
The congruences \eqref{R233n} and \eqref{R2327n} complete the proof. \qed
\subsection*{Proof of Corollary \ref{C23}}
Thanks to \eqref{Tlu02}, it
suffices to show that $\overline{R_{2,3}}\left(3^\beta(3n+2)\right)\equiv0\pmod{3}$.   For this, we show, by induction on $\alpha$, that for all $\alpha\geq0$, we have 
 \begin{align}
 \label{C231}\overline{R_{2,3}}\left(3^{2\alpha+1}(3n+2)\right)&\equiv0\pmod{3},\\
 \label{C232}\overline{R_{2,3}}\left(3^{2\alpha+2}(3n+2)\right)&\equiv0\pmod{3}.
 \end{align}
 As the right-hand side of \eqref{R233nA} contains no terms in which the exponents of $q$ are of the form $3n+2$ or $9n+6$, we conclude that  \begin{align*}\overline{R_{2,3}}\left(3(3n+2)\right)&=\overline{R_{2,3}}(9n+6)\equiv0\pmod{3},\\
 \overline{R_{2,3}}\left(3(9n+6)\right)&=\overline{R_{2,3}}(27n+18)\equiv0\pmod{3},
 \end{align*}
    which proves the base case $\alpha=0$ for \eqref{C231} and \eqref{C232}, respectively. We now assume that 
$\overline{R_{2,3}}\left(3^{2\alpha+1}(3n+2)\right)\equiv0\pmod{3}$ for some $\alpha\geq0$. Then,
\begin{align*}
    \overline{R_{2,3}}\left(3^{2(\alpha+1)+1}(3n+2)\right)&=\overline{R_{2,3}}\left(27 
   \left(3^{2\alpha}(3n+2)\right)\right)\\
    &\equiv \overline{R_{2,3}}\left(3 
   \left(3^{2\alpha}(3n+2)\right)\right) \pmod{3} \quad (\text{by Theorem \ref{T23}})\\
   &=\overline{R_{2,3}}
   \left(3^{2\alpha+1}(3n+2)\right)\\
   &\equiv0\pmod{3}.
\end{align*}
This completes the induction argument for \eqref{C231}.

Similarly, if $\overline{R_{2,3}}\left(3^{2\alpha+2}(3n+2)\right)\equiv0\pmod{3}$ for some $\alpha\geq0$, then we get 
\begin{align*}
    \overline{R_{2,3}}\left(3^{2(\alpha+1)+2}(3n+2)\right)&\equiv \overline{R_{2,3}}\left(3 
   \left(3^{2\alpha+1}(3n+2)\right)\right) \pmod{3}\\
   &\equiv0\pmod{3},
\end{align*}
completing the induction argument for \eqref{C232}.\qed
\section{proof of theorem \ref{T432n}}
We start with the generating function for $\overline{R_{4,3}}(n)$,
\begin{align*}
\sum_{n\geq 0}\overline{R_{4,3}}(n)q^n&=\frac{f_2f_4^2f_3^2f_{24}}{f_1^2f_{8}f_{6}f_{12}^2}\\
&=\left(\frac{f_4^4f_6f_{12}^2}{f_2^5f_8f_{24}}+2q\frac{f_4f_6^2f_8f_{24}}{f_2^4f_{12}}\right)\frac{f_2f_4^2f_{24}}{f_{8}f_{6}f_{12}^2} \quad \text{(thanks to \eqref{f32/f12})}.
\end{align*}
Extracting the terms that contain even powers of $q$ and replacing $q^2$ by $q$, we get
\begin{align}
 \label{432nvarphi} \sum_{n\geq 0}\overline{R_{4,3}}(2n)q^n&=\frac{f_2^6}{f_1^4f_4^2}=\frac{f_2^5}{f_1^2f_4^2}\left(\frac{f_2}{f_1^2}\right)=\frac{\varphi(q)}{\varphi(-q)} \quad \text{(thanks to \eqref{f25/f12f42}, \eqref{sf12/f2})}.
 \end{align}
Using \eqref{phiq^418} and \eqref{1/phi-q} in \eqref{432nvarphi}, we get
 \begin{align*}
  \sum_{n\geq 0}\overline{R_{4,3}}(2n)q^n&\equiv \varphi(q)^2\varphi(q^2)^2\pmod{8} \quad\\
&=\Big(1+2\sum_{n\geq1}q^{n^2}\Big)^2\Big(1+2\sum_{n\geq1}q^{2n^2}\Big)^2\\
&=\Big(1+4\sum_{n\geq1}q^{n^2}+4\sum_{m,n\geq1}q^{m^2+n^2}\Big)\\
&\quad\times\Big(1+4\sum_{n\geq1}q^{2n^2}+4\sum_{m,n\geq1}q^{2m^2+2n^2}\Big)\\
   &\equiv 1+4\sum_{n\geq1}q^{n^2}+4\sum_{m,n\geq1}q^{m^2+n^2}\\
   &\quad+4\sum_{n\geq1}q^{2n^2}+4\sum_{m,n\geq1}q^{2m^2+2n^2}\pmod{8}.\end{align*}
Observe that 
$$
\sum_{m,n\geq1}q^{m^2+n^2}=\sum_{n\geq1}q^{2n^2}+\sum_{\substack{n\neq m\\m,n\geq1}}q^{m^2+n^2}=\sum_{n\geq1}q^{2n^2}+2\sum_{m>n\geq1}q^{m^2+n^2},$$
and, similarly,
$$\sum_{m,n\geq1}q^{2m^2+2n^2}= \sum_{n\geq1}q^{4n^2}+2\sum_{m>n\geq1}q^{2m^2+2n^2}.
$$
Putting it all together gives
\begin{align*}
   \sum_{n\geq 0}\overline{R_{4,3}}(2n)q^n&\equiv1+ 4\sum_{n\geq1}q^{n^2}+4\sum_{n\geq1}q^{4n^2} \pmod{8}\\
    &=1+ 4\sum_{\text{odd }n\geq1}q^{n^2}+ 4\sum_{n\geq1}q^{4n^2}+4\sum_{n\geq1}q^{4n^2}\\
   &\equiv1+ 4\sum_{\text{odd }n\geq1}q^{n^2}\pmod{8}.
\end{align*}
The congruences \eqref{432n04} and \eqref{432n08} follow immediately. \qed
\begin{remark}
\label{remark:ppo}
Thanks to \eqref{pponvarphi} and \eqref{432nvarphi}, we have 
\begin{equation*}
    \sum_{n\geq 0}\overline{R_{4,3}}(2n)q^n=\sum_{n\geq 0}\overline{pp_o}(n)q^n,
\end{equation*}
that is, 
for all $n\geq 0$,
$$\overline{R_{4,3}}(2n)=\overline{pp_o}(n).$$ 
It follows that $\overline{R_{4,3}}(2n)$ inherits numerous arithmetic properties given the congruences which are already known for $\overline{pp_o}(n)$. 
In particular, from Lemma \ref{ppo2^k}, for all $n\geq 0$ and $k\geq4$, we have 
    \begin{align*}
        \overline{R_{4,3}}(2^{k+2}n)&\equiv0\pmod{2^{3k+2}}, \quad n\geq1,\\
        \overline{R_{4,3}}(2^{k+1}(4n+3))&\equiv0\pmod{2^{3k+1}},\\
        \overline{R_{4,3}}(2^{k+1}(8n+5))&\equiv0\pmod{2^{3k}}.
    \end{align*}
Additional arithmetic properties satisfied by $\overline{pp_o}(n)$ can be found in \cite{Lin}.
\end{remark}

\section{Proof of Theorem \ref{T494n012} and Theorem \ref{T493n08}}
\subsection*{Proof of Theorem \ref{T494n012}} It follows from \eqref{49n04} that, for all $n\geq 1$, $\overline{R_{4,9}}(4n)\equiv0\pmod{4}$. So, we only need to show that, for all $n\geq 1$,  $\overline{R_{4,9}}(4n)\equiv0\pmod{3}$. For this, we begin with the generating function
\begin{align}
   \label{49G} \sum_{n\geq 0}\overline{R_{4,9}}(n)q^n&=\frac{f_2f_4^2f_9^2f_{72}}{f_1^2f_{8}f_{18}f_{36}^2}\\
    \notag&\equiv \left(\frac{f_3^3}{f_1}\right)^2\frac{f_2f_4^2f_{72}}{f_8f_{18}f_{36}^2}\pmod{3} \quad \text{(thanks to \eqref{modp^k})}\\
    \notag&= \left(\frac{f_4^6f_6^4}{f_2^4f_{12}^2}+2q\frac{f_4^2f_6^2f_{12}^2}{f_2^2}+q^2\frac{f_{12}^6}{f_4^2}\right)\frac{f_2f_4^2f_{72}}{f_8f_{18}f_{36}^2} \quad \text{(using \eqref{f33/f1})}.
\end{align}
Extracting the terms containing even powers of $q$ and replacing $q^2$ by $q$ gives
\begin{align*}
    \sum_{n\geq0} \overline{R_{4,9}}(2n)q^{n}&\equiv\left(\frac{f_2^6f_3^4}{f_1^4f_{6}^2}+q\frac{f_{6}^6}{f_2^2}\right)\frac{f_1f_2^2f_{36}}{f_4f_9f_{18}^2} \pmod{3}\\
    &=\frac{f_2^8f_3^4f_{36}}{f_1^3f_4f_6^2f_9f_{18}^2} + q\frac{f_1f_6^6f_{36}}{f_4f_9f_{18}^2}\\
    &\equiv\left(\frac{f_3}{f_1^3}\right)\frac{f_2^8f_{36}}{f_4f_6^2f_{18}^2}+q\left(\frac{f_1}{f_3^3}\right)\frac{f_6^6f_{36}}{f_4f_{18}^2} \pmod{3} \quad \text{(thanks to \eqref{modp^k})}\\
    &= \left(\frac{f_4^6f_6^3}{f_2^9f_{12}^2}+3q\frac{f_4^2f_6f_{12}^2}{f_2^7}\right)\frac{f_2^8f_{36}}{f_4f_6^2f_{18}^2}+q\left(\frac{f_2f_4^2f_{12}^2}{f_6^7}-q\frac{f_2^3f_{12}^6}{f_4^2f_6^9}\right) \frac{f_6^6f_{36}}{f_4f_{18}^2},
\end{align*}
where the last equality follows using \eqref{f3/f13} and \eqref{f1/f33}. Again, extracting the terms that contain even powers of $q$ and replacing $q^2$ by $q$, we get 
\begin{align*}
     \sum_{n\geq0} \overline{R_{4,9}}(4n)q^{n}&\equiv\frac{f_2^5f_3f_{18}}{f_1f_6^2f_9^2}-q\frac{f_1^3f_6^6f_{18}}{f_2^3f_3^3f_9^2} \pmod{3}\\
     &\equiv \frac{f_2^8}{f_1^{16}}-q\frac{f_2^{24}}{f_1^{24}}\pmod{3}   \quad \text{(thanks to \eqref{modp^k})}.
\end{align*}
In order to complete the proof, it suffices to show that 
\begin{equation*}
    \frac{f_2^8}{f_1^{16}}-q\frac{f_2^{24}}{f_1^{24}}\equiv1\pmod{3}.
\end{equation*}
Equivalently, we show $G(q):=f_1^8f_2^8-qf_2^{24}-f_1^{24}\equiv0\pmod{3}$. For this, we rewrite the series $G$ as 
\begin{align*}
   G(q)&=\frac{f_1^9}{f_1}f_2^8-qf_2^{24}-\frac{f_1^{27}}{f_1^3} \\
   &\equiv\frac{f_3^3}{f_1}f_2^8-qf_2^{24}-\frac{f_9^3}{f_3}\pmod{3} \quad \text{(thanks to \eqref{modp^k})}\\
  &=\left(\frac{f_4^3f_6^2}{f_2^2f_{12}}+q\frac{f_{12}^3}{f_4}\right)f_2^8-qf_2^{24}-\left(\frac{f_{12}^3f_{18}^2}{f_{6}^2f_{36}}+q^3\frac{f_{36}^3}{f_{12}}\right) \quad \text{(using \eqref{f33/f1})}\\
  &=\frac{f_4^3f_6^2f_2^6}{f_{12}}+q\frac{f_2^8f_{12}^3}{f_4}-qf_2^{24}-\frac{f_{12}^3f_{18}^2}{f_{6}^2f_{36}}-q^3\frac{f_{36}^3}{f_{12}}\\
  &\equiv f_2^{12}+qf_2^8f_4^8-qf_2^{24}-f_2^{12}-q^3f_4^{24} \pmod{3}\\
  &=q(f_2^8f_4^8-q^2f_4^{24}-f_2^{24})\\
  &=qG(q^2).
\end{align*}
Lemma \ref{Gq=qGq} completes the proof. \qed
\subsection*{Proof of Theorem \ref{T493n08}}

Applying \eqref{f2/f12} and \eqref{f12/f2}
to \eqref{49G}, we obtain
\begin{align}\label{49n}
\sum_{n\ge 0}\overline{R_{4,9}}(n)q^{n} 
	= \left(\frac{f_6^4 f_9^6}{f_3^8 f_{18}^3} + 2q\, \frac{f_6^3 f_9^3}{f_3^7} + 4q^2\, \frac{f_6^2 f_{18}^3}{f_3^6}\right)
	\left(\frac{f_{36}^2}{f_{72}} - 2q^4\, \frac{f_{12} f_{72}^2}{f_{24} f_{36}}\right)
	\frac{f_9^2 f_{72}}{f_{18} f_{36}^2}.
\end{align}
Extracting the terms in which the exponents of $q$ are of the form $3n$, we get 
\begin{align*}
    \sum_{n\geq0}\overline{R_{4,9}}(3n)q^{3n}=& \frac{f_6^4f_9^8}{f_3^8f_{18}^4} - 8q^6 \frac{f_6^2 f_9^2 f_{12} f_{18}^2 f_{72}^3}{f_3^6 f_{24} f_{36}^3}.
\end{align*}
We now replace $q^3$ by $q$ to obtain
\begin{align}
   \label{493n} \sum_{n\geq0}\overline{R_{4,9}}(3n)q^{n}&=\frac{f_2^4 f_3^8}{f_1^8 f_6^4} - 8q^2 \frac{f_2^2 f_3^2 f_4 f_6^2 f_{24}^3}{f_1^6 f_8 f_{12}^3}\\
  \notag & \equiv\frac{f_2^4f_3^8}{f_1^8f_{6}^4}\pmod{8}\\
   \notag &\equiv\frac{f_1^8f_3^8}{f_1^8f_3^8}\pmod{8} \quad \text{(thanks to \eqref{modp^k})}\\
   \notag &=1.
\end{align}
This completes the proof.\qed

\section{proof of theorem \ref{T494knr}}
It follows from \eqref{494n012} that, for all $n\geq 0$,$\overline{R_{4,9}}(4(kn+r))\equiv0\pmod{3}$. Thus, we only need to show that, for all $n\geq 0$, $\overline{R_{4,9}}(4(kn+r))\equiv0\pmod{8}$. We have,
\begin{align*} 
\sum_{n\geq 0}\overline{R_{4,9}}(n)q^n&=\frac{f_2f_4^2f_9^2f_{72}}{f_1^2f_{8}f_{18}f_{36}^2} \\
&=\frac{\varphi(-q^4)\varphi(-q^9)}{\varphi(-q)\varphi(-q^{36})} \quad \text{(thanks to \eqref{sf12/f2})}\\
&\equiv \frac{\varphi(q)\varphi(q^2)^2\varphi(q^{36})\varphi(q^{72})^2}{\varphi(q^4)\varphi(q^8)^2\varphi(q^9)\varphi(q^{18})^2}\pmod{8} \quad \text{(thanks to \eqref{phiq^418} and \eqref{1/phi-q})}\\
&\equiv \varphi(q)\varphi(q^2)^2\varphi(q^{36})\varphi(q^{72})^2 \varphi(-q^4)\varphi(-q^9)\pmod{8} \quad \text{(thanks to \eqref{phiq})}\\
&=\Big(1+2\sum_{n\geq1}q^{n^2}\Big)\Big(1+2\sum_{n\geq1}q^{2n^2}\Big)^2\Big(1+2\sum_{n\geq1}q^{36n^2}\Big)\Big(1+2\sum_{n\geq1}q^{72n^2}\Big)^2\\
&\quad \times \Big(1+2\sum_{n\geq1}(-1)^nq^{4n^2}\Big)\Big(1+2\sum_{n\geq1}(-1)^nq^{9n^2}\Big)\\
&=\Big(1+2\sum_{n\geq1}q^{n^2}\Big)\Big(1+4\sum_{n\geq1}q^{2n^2}+4\sum_{m,n\geq1}q^{2m^2+2n^2}\Big) \\
&\quad \times \Big(1+2\sum_{n\geq1}q^{36n^2}\Big)
\Big(1+4\sum_{n\geq1}q^{72n^2}+4\sum_{m,n\geq1}q^{72m^2+72n^2}\Big)\\
&\quad \times \Big(1+2\sum_{n\geq1}(-1)^nq^{4n^2}\Big)\Big(1+2\sum_{n\geq1}(-1)^nq^{9n^2}\Big)\\
\end{align*}
By further expanding the right-hand side of the above congruence and reducing it modulo 8, we get
\begin{align}\label{49nvarphi}
\sum_{n\geq 0}\overline{R_{4,9}}(n)q^n
&\equiv1+2\sum_{n\geq1}q^{n^2}+4\sum_{n\geq1}q^{2n^2}+2\sum_{n\geq1}q^{36n^2}+4\sum_{n\geq1}q^{72n^2}\\
\notag&\quad +2\sum_{n\geq1}(-1)^nq^{4n^2}
 +2\sum_{n\geq1}(-1)^nq^{9n^2}+4\sum_{m,n\geq1}q^{2m^2+2n^2}\\
\notag&\quad+4\sum_{m,n\geq1}q^{72m^2+72n^2}+4\sum_{m,n\geq1}q^{m^2+36n^2}+4\sum_{m,n\geq1}q^{9m^2+36n^2}\\
\notag&\quad+4\sum_{m,n\geq1}q^{m^2+9n^2}+4\sum_{m,n\geq1}q^{4m^2+9n^2}+4\sum_{m,n\geq1}q^{m^2+4n^2}\\
\notag&\quad +4\sum_{m,n\geq1}q^{4m^2+36n^2} \pmod{8}.
\end{align}

We now observe the following.
\begin{align}
  \label{B1} \sum_{n\geq1}q^{n^2}&=  \sum_{n\geq1}q^{4n^2}+ \sum_{\text{odd } n\geq1}q^{n^2},\\
\sum_{n\geq1}q^{2n^2}&=\sum_{n\geq1}q^{8n^2}+\sum_{\text{odd }n\geq1}q^{2n^2}\\
   \label{B3}\sum_{n\geq1}(-1)^nq^{9n^2}&=\sum_{n\geq1}(-1)^nq^{36n^2}+\sum_{\text{odd }n\geq1}(-1)^nq^{9n^2},\\
\sum_{m,n\geq1}q^{2m^2+2n^2}&=\sum_{n\geq1}q^{4n^2}+2\sum_{m>n\geq1}q^{2m^2+2n^2},\\
\sum_{m,n\geq1}q^{72m^2+72n^2}&=\sum_{n\geq1}q^{144n^2}+2\sum_{m>n\geq1}q^{72m^2+72n^2},\\
\sum_{m,n\geq1}q^{m^2+36n^2}&=\sum_{m,n\geq1}q^{4m^2+36n^2}+\sum_{\substack{m,n\geq1\\ m \text{ odd}}}q^{m^2+36n^2}\\
\sum_{m,n\geq1}q^{m^2+9n^2}&=\sum_{m,n\geq1}q^{4m^2+9n^2}+\sum_{\substack{m,n\geq1\\m\text{ odd}}}q^{m^2+9n^2}\\
\notag&=\sum_{m,n\geq1}q^{4m^2+9n^2}+\sum_{\substack{m,n\geq1\\m,n \text{ odd}}}q^{m^2+(3n)^2}+\sum_{\substack{m,n\geq1\\m \text{ odd}}}q^{m^2+36n^2},\\
\sum_{m,n\geq1}q^{m^2+4n^2}&=\sum_{m,n\geq1}q^{4m^2+4n^2}+\sum_{\substack{m,n\geq1\\m \text{ odd}}}q^{m^2+4n^2}\\
\notag &=\sum_{n\geq1}q^{8n^2}+2\sum_{m>n\geq1}q^{4m^2+4n^2}+\sum_{\substack{m,n\geq1\\m \text{ odd}}}q^{m^2+4n^2},\\
\label{B2}\sum_{m,n\geq1}q^{9m^2+36n^2}&=\sum_{m,n\geq1}q^{36m^2+36n^2}+\sum_{\substack{m,n\geq1\\m \text{ odd}}}q^{9m^2+36n^2}\\
\notag&=\sum_{n\geq1}q^{72n^2}+2\sum_{m>n\geq1}q^{36m^2+36n^2}+\sum_{\substack{m,n\geq1\\m \text{ odd}}}q^{9m^2+36n^2}.
\end{align}
Using \eqref{B1}--\eqref{B2} in \eqref{49nvarphi}, reducing modulo 8, and then extracting the terms in which the exponents of $q$ are of the form $4n$, we get
\begin{align*}
    \sum_{n\geq0}\overline{R_{4,9}}(4n)q^{4n}&\equiv6\sum_{n\geq1}q^{4n^2}+2\sum_{n\geq1}q^{36n^2}+2\sum_{n\geq1}(-1)^nq^{4n^2}+2\sum_{n\geq1}(-1)^nq^{36n^2}\\
    &\quad+4\sum_{n\geq1}q^{144n^2} \pmod{8}.
\end{align*}
Replacing $q^4$ by $q$ yields 
\begin{align*}
    \sum_{n\geq0}\overline{R_{4,9}}(4n)q^{n}&\equiv6\sum_{n\geq1}q^{n^2}+2\sum_{n\geq1}q^{(3n)^2}+2\sum_{n\geq1}(-1)^nq^{n^2}+2\sum_{n\geq1}(-1)^nq^{(3n)^2}\\
    &\quad +4\sum_{n\geq1}q^{(6n)^2} \pmod{8}.
\end{align*}
Notice that every exponent of $q$ on the right-hand side sum is a square number, and if $m^2=kn+r$, then $m^2\equiv r\pmod{k}$. But $r$ is a quadratic nonresidue modulo $k$. Hence,
$$\sum_{n\geq0}\overline{R_{4,9}}(4(kn+r))q^{kn+r}\equiv0\pmod{8}.$$
This proves the congruence \eqref{494knr08}.

In order to prove \eqref{49n04}, we reduce \eqref{49nvarphi} modulo 4 and obtain
\begin{align*}
\sum_{n\geq 0}\overline{R_{4,9}}(n)q^n
&\equiv1+2\sum_{n\geq1}q^{n^2}+2\sum_{n\geq1}q^{36n^2}+2\sum_{n\geq1}q^{4n^2}+2\sum_{n\geq1}q^{9n^2} \pmod{4}\\
&\equiv1+2\sum_{\text{odd }n\geq1}q^{n^2}+2\sum_{\text{odd }n\geq1}q^{9n^2}\pmod{4},
\end{align*}
where the last congruence follows using \eqref{B1} and \eqref{B3}.
This completes the proof.\qed
\section{Proof of Theorem \ref{T49}}
\subsection*{Proof of $\eqref{4918n1296}$}
From \eqref{493n},
\begin{align*}
\sum_{n \ge 0} \overline{R_{4,9}}(3n) q^n 
		&= \frac{f_2^4 f_3^8}{f_1^8 f_6^4} - 8q^2 \frac{f_2^2 f_3^2 f_4 f_6^2 f_{24}^3}{f_1^6 f_8 f_{12}^3} \\
		&\equiv \left( \frac{f_2}{f_1^2} \right)^4 \frac{f_3^8}{f_6^4} - 8q^2 \frac{(f_1 f_2)^2}{f_4 f_8} \frac{f_3^2 f_6^2 f_{24}^3}{f_{12}^3} \pmod{32} \quad \text{(by \eqref{modp^k})} \\
		&= \left( \frac{f_6^4 f_9^6}{f_3^8 f_{18}^3} + 2q \frac{f_6^3 f_9^3}{f_3^7} + 4q^2 \frac{f_6^2 f_{18}^3}{f_3^6} \right)^4 \frac{f_3^8}{f_6^4} \\
		&\quad - 8q^2 \left( \frac{f_6 f_9^4}{f_3 f_{18}^2} - q f_9 f_{18} - 2q^2 \frac{f_3 f_{18}^4}{f_6 f_9^2} \right)^2 \\
		&\quad \times \left( \frac{f_{36}^9}{f_{12}^6 f_{24}^2 f_{72}^3} + q^4 \frac{f_{36}^6}{f_{12}^5 f_{24}^3} + 3q^8 \frac{f_{36}^3 f_{72}^3}{f_{12}^4 f_{24}^4} \right. \\
		&\quad \left. - 2q^{12} \frac{f_{72}^6}{f_{12}^3 f_{24}^5} + 4q^{16} \frac{f_{72}^9}{f_{12}^2 f_{24}^6 f_{36}^3} \right) \frac{f_3^2 f_6^2 f_{24}^3}{f_{12}^3}.
\end{align*}
The last equality is derived using \eqref{f2/f12}, \eqref{f1f2}, and \eqref{1/f1f2}. Extracting the terms where the exponents of $q$ are congruent to 1 modulo 3 gives
\begin{equation*}
	\begin{aligned}
		\sum_{n \ge 0} \overline{R_{4,9}}(9n+3) q^{3n+1} 
		&\equiv 8q \frac{f_6^{11} f_9^{21}}{f_3^{23} f_{18}^9} - 8q^4 \frac{f_3^2 f_6^2 f_9^2 f_{18}^2 f_{24} f_{36}^9}{f_{12}^9 f_{72}^3} + 16q^4 \frac{f_6^8 f_9^{12}}{f_3^{20}} \\
		&\quad + 16q^7 \frac{f_3 f_6^3 f_9^5 f_{36}^6}{f_{12}^8 f_{18}} + 8q^{10} \frac{f_6^4 f_9^8 f_{36}^3 f_{72}^3}{f_{12}^7 f_{18}^4 f_{24}} \\
		&\quad + 16q^{16} \frac{f_3^2 f_6^2 f_9^2 f_{18}^2 f_{72}^6}{f_{12}^6 f_{24}^2} \pmod{32}.
	\end{aligned}
\end{equation*}
After dividing the series by $q$ and replacing $q^3$ by $q$, the expression transforms into
\begin{align*}
		\sum_{n \ge 0} \overline{R_{4,9}}(9n+3) q^n 
		&\equiv 8 \frac{f_2^{11} f_3^{21}}{f_1^{23} f_{6}^9} - 8q \frac{f_1^2 f_2^2 f_3^2 f_6^2 f_8 f_{12}^9}{f_4^9 f_{24}^3} + 16q \frac{f_2^{8} f_3^{12}}{f_1^{20}} \\
		&\quad + 16q^2 \frac{f_1 f_2^3 f_3^5 f_{12}^6}{f_4^8 f_6} + 8q^3 \frac{f_2^4 f_3^8 f_{12}^3 f_{24}^3}{f_4^7 f_6^4 f_8} \\
		&\quad + 16q^5 \frac{f_1^2 f_2^2 f_3^2 f_6^2 f_{24}^6}{f_4^6 f_8^2} \pmod{32} \\
		&\equiv 8 f_1 f_3 \frac{f_6}{f_2} - 8q (f_1 f_3)^2 \frac{f_2^2 f_6^2 f_8 f_{12}^9}{f_4^9 f_{24}^3} + 16q \frac{f_6^6}{f_2^2} \\
		&\quad + 16q^2 f_1 f_3 \frac{f_2^3 f_6 f_{12}^6}{f_4^8} + 8q^3 \frac{f_2^4 f_{12}^3 f_{24}^3}{f_4^7 f_8} + 16q^5 \frac{f_2^3 f_6^3 f_{24}^6}{f_4^6 f_8^2} \pmod{32} \\
		&= 8 \left( \frac{f_2 f_8^2 f_{12}^4}{f_4^2 f_6 f_{24}^2} - q \frac{f_4^4 f_6 f_{24}^2}{f_2 f_8^2 f_{12}^2} \right) \frac{f_6}{f_2} \\
		&\quad - 8q \left( \frac{f_2 f_8^2 f_{12}^4}{f_4^2 f_6 f_{24}^2} - q \frac{f_4^4 f_6 f_{24}^2}{f_2 f_8^2 f_{12}^2} \right)^2 \frac{f_2^2 f_6^2 f_8 f_{12}^9}{f_4^9 f_{24}^3} \\
		&\quad + 16q \frac{f_6^6}{f_2^2} + 16q^2 \left( \frac{f_2 f_8^2 f_{12}^4}{f_4^2 f_6 f_{24}^2} - q \frac{f_4^4 f_6 f_{24}^2}{f_2 f_8^2 f_{12}^2} \right) \frac{f_2^3 f_6 f_{12}^6}{f_4^8} \\
		&\quad + 8q^3 \frac{f_2^4 f_{12}^3 f_{24}^3}{f_4^7 f_8} + 16q^5 \frac{f_2^3 f_6^3 f_{24}^6}{f_4^6 f_8^2} \quad \text{(by \eqref{f1f2})}
\end{align*}
By isolating the terms that contain odd powers of $q$, we get
\begin{align*}
		\sum_{n \ge 0} \overline{R_{4,9}}(18n+12) q^{2n+1} 
		&\equiv -8q \frac{f_4^4 f_6^2 f_{24}^2}{f_2^2 f_8^2 f_{12}^2} - 8q \frac{f_2^4 f_8^5 f_{12}^{17}}{f_4^{13} f_{24}^7} - 8q^3 \frac{f_6^4 f_{12}^5 f_{24}}{f_4 f_8^3} \\
		&\quad + 16q \frac{f_6^6}{f_2^2} - 16q^3 \frac{f_2^2 f_6^2 f_{12}^4 f_{24}^2}{f_4^4 f_8^2} + 8q^3 \frac{f_2^4 f_{12}^3 f_{24}^3}{f_4^7 f_8} \\
		&\quad + 16q^5 \frac{f_2^3 f_6^3 f_{24}^6}{f_4^6 f_8^2} \pmod{32} \\
		&\equiv -8q \frac{f_6^6}{f_2^2} - 8q \frac{f_4 f_{12} f_{24}}{f_8} - 8q^3 \frac{f_6^4 f_{12}^5 f_{24}}{f_4 f_8^3} + 16q \frac{f_6^6}{f_2^2} \\
		&\quad - 16q^3 \frac{f_6^{18}}{f_2^{14}} + 8q^3 \frac{f_6^4 f_{12}^5 f_{24}}{f_4 f_8^3} + 16q^5 \frac{f_6^{27}}{f_2^{17}} \pmod{32} \quad \text{(by \eqref{modp^k})} \\
		&= 8q \frac{f_6^6}{f_2^2} - 8q \frac{f_4 f_{12} f_{24}}{f_8} - 16q^3 \frac{f_6^{18}}{f_2^{14}} + 16q^5 \frac{f_6^{27}}{f_2^{17}}.
\end{align*}
Dividing both sides by $q$ and replacing $q^2$ by $q$ yields
\begin{align*}
		\sum_{n \ge 0} \overline{R_{4,9}}(18n+12) q^n 
		&\equiv 8 \frac{f_3^6}{f_1^2} - 8 \frac{f_2 f_{6} f_{12}}{f_4} - 16q \frac{f_3^{18}}{f_1^{14}} + 16q^2 \frac{f_3^{27}}{f_1^{17}} \pmod{32} \\
		&\equiv 8 \frac{f_3^2}{f_1^2} f_6^2 - 8 \frac{f_2 f_{6} f_{12}}{f_4} - 16q \frac{f_3^{18}}{f_1^{14}} + 16q^2 \frac{f_3}{f_1} \frac{f_6^{13}}{f_2^{8}} \pmod{32} \\
		&\quad \text{(by \eqref{modp^k})} \\
		&= 8 \left( \frac{f_4^4 f_6 f_{12}^2}{f_2^5 f_8 f_{24}} + 2q \frac{f_4 f_6^2 f_8 f_{24}}{f_2^4 f_{12}} \right) f_6^2 - 8 \frac{f_2 f_{6} f_{12}}{f_4} \\
		&\quad - 16q \frac{f_3^{18}}{f_1^{14}} + 16q^2 \left( \frac{f_4 f_6 f_{16} f_{24}^2}{f_2^2 f_8 f_{12} f_{48}} + q \frac{f_6 f_8^2 f_{48}}{f_2^2 f_{16} f_{24}} \right) \frac{f_6^{13}}{f_2^{8}} \\
		&\quad \text{(by \eqref{f32/f12}, \eqref{f3/f1})} \\
		&= 8 \frac{f_4^4 f_6^3 f_{12}^2}{f_2^5 f_8 f_{24}} + 16q \frac{f_4 f_6^4 f_8 f_{24}}{f_2^4 f_{12}} - 8 \frac{f_2 f_6 f_{12}}{f_4} - 16q \frac{f_3^{18}}{f_1^{14}} \\
		&\quad + 16q^2 \frac{f_4 f_6^{14} f_{16} f_{24}^2}{f_2^{10} f_8 f_{12} f_{48}} + 16q^3 \frac{f_6^{14} f_8^2 f_{48}}{f_2^{10} f_{16} f_{24}} \\
		&\equiv 8 \frac{f_6^7 f_8}{f_2^5 f_{24}} + 16q \frac{f_4 f_6^4 f_8 f_{24}}{f_2^4 f_{12}} - 8 \frac{f_2 f_6 f_{12}}{f_4} - 16q \frac{f_6^9}{f_2^7} \\
		&\quad + 16q^2 \frac{f_{12}^6 f_{16}}{f_8^3} + 16q^3 \frac{f_6^{14} f_8^2 f_{48}}{f_2^{10} f_{16} f_{24}} \pmod{32} \quad \text{(by \eqref{modp^k})}.
\end{align*}
Now, isolating the terms of even power and replacing $q^2$ by $q$, we obtain
\begin{align*}
		\sum_{n \ge 0} \overline{R_{4,9}}(36n+12) q^n 
		&\equiv 8 \frac{f_3^7 f_4}{f_1^5 f_{12}} - 8 \frac{f_1 f_3 f_{6}}{f_2} + 16q \frac{f_{6}^6 f_{8}}{f_4^3} \pmod{32} \\
		&\equiv 8 \frac{f_3^3}{f_1} \frac{f_4 f_6^2}{f_2^2 f_{12}} - 8 \frac{f_1 f_3 f_6}{f_2} + 16q \frac{f_6^6 f_8}{f_4^3} \pmod{32} \\
		&\quad \text{(by \eqref{modp^k})} \\
		&= 8 \left( \frac{f_4^3 f_6^2}{f_2^2 f_{12}} + q \frac{f_{12}^3}{f_4} \right) \frac{f_4 f_6^2}{f_2^2 f_{12}} \\
		&\quad - 8 \left( \frac{f_2 f_8^2 f_{12}^4}{f_4^2 f_6 f_{24}^2} - q \frac{f_4^4 f_6 f_{24}^2}{f_2 f_8^2 f_{12}^2} \right) \frac{f_6}{f_2} + 16q \frac{f_6^6 f_8}{f_4^3} \\
		&\quad \text{(by \eqref{f33/f1}, \eqref{f1f3})} \\
		&\equiv 8 f_2^4 + 8q \frac{f_6^2 f_{12}^2}{f_2^2} - 8 f_2^4 + 8q \frac{f_6^2 f_{12}^2}{f_2^2} + 16q \frac{f_6^2 f_{12}^2}{f_2^2} \pmod{32} \\
		&\quad \text{(by \eqref{modp^k})} \\
		&= 0.
\end{align*}
Similarly, when the terms of odd power are extracted, dividing by $q$ and replacing $q^2$ by $q$ leads to
\begin{align*}
		\sum_{n \ge 0} \overline{R_{4,9}}(36n+30) q^n 
		&\equiv 16 \frac{f_2 f_3^4 f_4 f_{12}}{f_1^4 f_6} + 16 \frac{f_3^9}{f_1^7} - 16q \frac{f_3^{14} f_4^2 f_{24}}{f_1^{10} f_8 f_{12}} \pmod{32} \\
		&\equiv 16 \frac{f_2 f_6^2 f_4 f_{12}}{f_2^2 f_6} + 16 \frac{f_3}{f_1} \frac{f_6^4}{f_2^3} - 16q \frac{f_6^7 f_4^2 f_{24}}{f_2^5 f_8 f_{12}} \pmod{32} \\
		&\quad \text{(by \eqref{modp^k})} \\
		&= 16 \frac{f_2 f_6^2 f_4 f_{12}}{f_2^2 f_6} + 16 \left( \frac{f_4 f_6 f_{16} f_{24}^2}{f_2^2 f_8 f_{12} f_{48}} + q \frac{f_6 f_8^2 f_{48}}{f_2^2 f_{16} f_{24}} \right) \frac{f_6^4}{f_2^3} \\
		&\quad - 16q \frac{f_6^7 f_4^2 f_{24}}{f_2^5 f_8 f_{12}} \pmod{32} \quad \text{(by \eqref{f3/f1})} \\
		&\equiv 16 f_2 f_6^3 + 16 f_2 f_6^3 + 16q \frac{f_6^9}{f_2^5} - 16q \frac{f_6^9}{f_2^5} \pmod{32} \\
		&\quad \text{(by \eqref{modp^k})} \\
		&= 0.
\end{align*}
Therefore, we have established that
\begin{equation*}
	\overline{R_{4,9}}(18n+12)\equiv 0 \pmod{32}.
\end{equation*}
For an elementary proof of $\overline{R_{4,9}}(18n+12)\equiv 0 \pmod{3}$, the reader is referred to \cite[Theorem 5.2]{Ghoshal-Jana}. This completes the proof of \eqref{4918n1296}.

\subsection*{Proof of \eqref{4918n1548}}
For an elementary proof of $\overline{R_{4,9}}(18n+15)\equiv 0 \pmod{3} $, see \cite[Theorem 5.2]{Ghoshal-Jana}. It remains to show $\overline{R_{4,9}}(18n+15)\equiv 0 \pmod{16}$. By applying \eqref{modp^k} to \eqref{493n}, we get
\begin{align*}
 \sum_{n\geq0}\overline{R_{4,9}}(3n)q^{n}&\equiv\frac{f_2^4 f_3^8}{f_1^8 f_6^4} - 8q^2 \frac{f_4 f_6 f_4 f_{12} f_{12}^6}{f_2^3 f_4^2 f_{12}^3} \pmod{16}\\
 &=\left(\frac{f_3^2}{f_1^2}\right)^4\frac{f_2^4}{f_6^4}-8q^2\frac{f_6f_{12}^4}{f_2^3}\\
 &=\left(\frac{f_4^4f_6f_{12}^2}{f_2^5f_8f_{24}}+2q\frac{f_4f_6^2f_8f_{24}}{f_2^4f_{12}}\right)^4\frac{f_2^4}{f_6^4}-8q^2\frac{f_6f_{12}^4}{f_2^3} \quad \text{(applying \eqref{f32/f12})}\\
 &\equiv\left\{\left(\frac{f_4^4f_6f_{12}^2}{f_2^5f_8f_{24}}\right)^4+4\left(\frac{f_4^4f_6f_{12}^2}{f_2^5f_8f_{24}}\right)^3\left(2q\frac{f_4f_6^2f_8f_{24}}{f_2^4f_{12}}\right)\right.\\
&\left.+6\left(\frac{f_4^4f_6f_{12}^2}{f_2^5f_8f_{24}}\right)^2\left(2q\frac{f_4f_6^2f_8f_{24}}{f_2^4f_{12}}\right)^2\right\}\frac{f_2^4}{f_6^4}-8q^2\frac{f_6f_{12}^4}{f_2^3}\pmod{16}.
\end{align*}
Extracting the terms that contain odd powers of $q$ gives 
\begin{align*}
\sum_{n\geq0}\overline{R_{4,9}}(6n+3)q^{2n+1}\equiv 4\left(\frac{f_4^4f_6f_{12}^2}{f_2^5f_8f_{24}}\right)^3\left(2q\frac{f_4f_6^2f_8f_{24}}{f_2^4f_{12}}\right)\frac{f_2^4}{f_6^4}\pmod{16}.
\end{align*}
Dividing by $q$ and then replacing $q^2$ by $q$ yields
\begin{align*}
   \sum_{n\geq0}\overline{R_{4,9}}(6n+3)q^n&\equiv8 \frac{f_2^{13}f_3f_6^5}{f_1^{15}f_4^2f_{12}^2}\pmod{16}\\
   &\equiv8\frac{f_1^{26}f_3f_6^5}{f_1^{15}f_1^8f_{12}^2} \pmod{16} \quad\text{(thanks to \eqref{modp^k})}\\
   &=8f_1^3\frac{f_3f_6^5}{f_{12}^2}\\
   &=8\frac{f_3f_6^5}{f_{12}^2}\sum_{m\geq0}(-1)^m(2m+1)q^{m(m+1)/2} \quad \text{(using \eqref{f13})}.
\end{align*}
If $m(m+1)/2=3n+2$, then $(2m+1)^2=24n+17$. But $17$ is a quadratic nonresidue modulo $24$. Thus, we have $m(m+1)/2\neq3n+2$ for any integers $m,n$. From this we conclude that the right-hand side of the last congruence does not contain terms in which the exponents of $q$ are of the form $3n+2$. Therefore, we must have $$\sum_{n\geq0}\overline{R_{4,9}}(18n+15)q^{3n+2}\equiv0\pmod{16}.$$
This completes the proof.

\subsection*{Proof of \eqref{4924n20216}}
Observe that $\overline{R_{4,9}}(24n+20)\equiv0\pmod{8}$ follows from \eqref{494knr08}. To show $\overline{R_{4,9}}(24n+20)\equiv0\pmod{27}$, we begin by extracting the terms in which the exponents of $q$ are of the form ${3n+2}$ in \eqref{49n} and obtain the following. 
\begin{align*}
    \sum_{n\geq0}\overline{R_{4,9}}(3n+2)q^{3n+2}=-4q^5\frac{f_6^3f_9^5f_{12}f_{72}^3}{f_3^7f_{18}f_{24}f_{36}^3}+4q^2\frac{f_6^2f_9^2f_{18}^2}{f_3^6}.
\end{align*}
Dividing by $q^3$ and replacing $q^3$ by $q$ gives
\begin{align*}
\sum_{n\geq0}\overline{R_{4,9}}(3n+2)q^n&=-4q\frac{f_2^3f_3^5f_{4}f_{24}^3}{f_1^7f_{6}f_{8}f_{12}^3}+4\frac{f_2^2f_3^2f_6^2}{f_1^6}\\
    &=\left\{-4q\left(\frac{f_3^3}{f_1}\right)\frac{f_2^3f_4f_{24}^3}{f_6f_8f_{12}^3}+4f_2^2f_6^2\right\}\left(\frac{f_3}{f_1^3}\right)^2\\
&=\left\{-4q\left(\frac{f_4^3f_6^2}{f_2^2f_{12}}+q\frac{f_{12}^3}{f_4} \right)\frac{f_2^3f_4f_{24}^3}{f_6f_8f_{12}^3}+4f_2^2f_6^2\right\}\\
&\quad\times\left(\frac{f_4^6f_6^3}{f_2^9f_{12}^2}+3q\frac{f_4^2f_6f_{12}^2}{f_2^7}\right)^2.
\end{align*}
Isolating the terms that contain even powers of $q$, we get
\begin{align*}
    \sum_{n\geq0}\overline{R_{4,9}}(6n+2)q^{2n}&=-4q\frac{f_4^3f_6^2}{f_2^2f_{12}}\frac{f_2^3f_4f_{24}^3}{f_6f_8f_{12}^3}\left(6q\frac{f_4^6f_6^3}{f_2^9f_{12}^2}\frac{f_4^2f_6f_{12}^2}{f_2^7}\right)\\
    &\quad -4q^2\frac{f_{12}^3}{f_4}\frac{f_2^3f_4f_{24}^3}{f_6f_8f_{12}^3}\left(\frac{f_4^{12}f_6^6}{f_2^{18}f_{12}^4}+9q^2\frac{f_4^4f_6^2f_{12}^4}{f_2^{14}}\right)\\
    &\quad + 4f_2^2f_6^2\left(\frac{f_4^{12}f_6^6}{f_2^{18}f_{12}^4}+9q^2\frac{f_4^4f_6^2f_{12}^4}{f_2^{14}}\right)\\
    &=-28q^2\frac{f_4^{12}f_6^5f_{24}^3}{f_2^{15}f_8f_{12}^4}-36q^4\frac{f_4^4f_6f_{12}^4f_{24}^3}{f_2^{11}f_8}+4\frac{f_4^{12}f_6^8}{f_2^{16}f_{12}^4}+36q^2\frac{f_4^4f_6^4f_{12}^4}{f_2^{12}}.
\end{align*}
Replacing $q^2$ by $q$ and reducing modulo 27, we get 
\begin{align*}
    \sum_{n\geq0}\overline{R_{4,9}}(6n+2)q^n&\equiv-q\frac{f_2^{12}f_3^5f_{12}^3}{f_1^{15}f_4f_{6}^4}-9q^2\frac{f_2^4f_3f_{6}^4f_{12}^3}{f_1^{11}f_4}+4\frac{f_2^{12}f_3^8}{f_1^{16}f_{6}^4}\\
    &\quad+9q\frac{f_2^4f_3^4f_{6}^4}{f_1^{12}}\pmod{27}.
\end{align*}
We apply \eqref{modp^k} to the second and fourth terms on the right-hand side to obtain
\begin{align*}
\sum_{n\geq0}\overline{R_{4,9}}(6n+2)q^n&\equiv-q\left(\frac{f_3}{f_1^3}\right)^5\frac{f_2^{12}f_{12}^3}{f_4f_6^4}-9q^2\left(\frac{1}{f_1^4}\right)^2\frac{f_2^4f_6^4f_{12}^3}{f_4}\\
&\quad +4\frac{f_3^3}{f_1}\left(\frac{f_3}{f_1^3}\right)^5\frac{f_2^{12}}{f_6^4}+9qf_2^{16}\pmod{27}\\
&=-q\left(\frac{f_4^6f_6^3}{f_2^9f_{12}^2}+3q\frac{f_4^2f_6f_{12}^2}{f_2^7}\right)^5\frac{f_2^{12}f_{12}^3}{f_4f_6^4}\\
&\quad-9q^2\left(\frac{f_4^{28}}{f_2^{28}f_8^8}+8q\frac{f_4^{16}}{f_2^{24}}+16q^2\frac{f_4^4f_8^8}{f_2^{20}}\right)\frac{f_2^4f_6^4f_{12}^3}{f_4}\\
&\quad+4\left(\frac{f_4^3f_6^2}{f_2^2f_{12}}+q\frac{f_{12}^3}{f_4}\right)\left(\frac{f_4^6f_6^3}{f_2^9f_{12}^2}+3q\frac{f_4^2f_6f_{12}^2}{f_2^7}\right)^5\frac{f_2^{12}}{f_6^4}+9qf_2^{16},
\end{align*}
where the last equality follows by using \eqref{f33/f1}, \eqref{f3/f13}, and \eqref{1/f14}. Observing $(x+3qy)^5\equiv x^5+15qx^4y+9q^2x^3y^2\pmod{27}$ and extracting the terms that contain odd powers of $q$ yields the following:
\begin{align*}
\sum_{n\geq0}\overline{R_{4,9}}(12n+8)q^{2n+1}&\equiv-q\left(\frac{f_4^{30}f_6^{15}}{f_2^{45}f_{12}^{10}}+9q^2\frac{f_4^{22}f_6^{11}}{f_2^{41}f_{12}^2}\right)\frac{f_2^{12}f_{12}^3}{f_4f_6^4}-72q^3\frac{f_4^{15}f_6^4f_{12}^3}{f_2^{20}}\\
&\quad+4q\frac{f_{12}^3}{f_4}\left(\frac{f_4^{30}f_6^{15}}{f_2^{45}f_{12}^{10}}+9q^2\frac{f_4^{22}f_6^{11}}{f_2^{41}f_{12}^2}\right)\frac{f_2^{12}}{f_6^4}\\
&\quad+4\frac{f_4^3f_6^2}{f_2^2f_{12}}\left(15q\frac{f_4^{26}f_6^{13}}{f_2^{43}f_{12}^6}\right)\frac{f_2^{12}}{f_6^4}+9qf_2^{16}\pmod{27}\\
&\equiv9q\frac{f_4^{29}f_6^{11}}{f_2^{33}f_{12}^7}+9q^3\frac{f_4^{15}f_6^4f_{12}^3}{f_2^{20}}+9qf_2^{16}\pmod{27}.
\end{align*}
We divide by $q$ and replace $q^2$ by $q$ to obtain
\begin{align*}
\sum_{n\geq0}\overline{R_{4,9}}(12n+8)q^n&\equiv 9\frac{f_2^{29}f_3^{11}}{f_1^{33}f_{6}^7}+9q\frac{f_2^{15}f_3^4f_{6}^3}{f_1^{20}}+9\left(\frac{f_1^9}{f_1}\right)^2\pmod{27}\\
&\equiv9\frac{f_2^{29}}{f_6^7}+9qf_2^{15}f_6^3\left(\frac{1}{f_1^4}\right)^2+9\left(\frac{f_9}{f_1}\right)^2\pmod{27}\\
&\quad\text{(thanks to  \eqref{modp^k} and a factor of 9 in each term)}\\
&=9\frac{f_2^{29}}{f_6^7}+9qf_2^{15}f_6^3\left(\frac{f_4^{28}}{f_2^{28}f_8^8}+8q\frac{f_4^{16}}{f_2^{24}}+16q^2\frac{f_4^4f_8^8}{f_2^{20}}\right)\\
&\quad+9\left(\frac{f_{12}^6f_{18}^2}{f_2^4f_6^2f_{36}^2}+2q\frac{f_4^2f_{12}^2f_{18}}{f_2^5}+q^2\frac{f_4^4f_6^2f_{36}^2}{f_2^6f_{12}^2}\right),
\end{align*}
where the last equality follows form \eqref{f9/f1} and \eqref{1/f14}. Finally, we extract the terms that contain odd powers of $q$, divide by $q$, and replace $q^2$ by $q$ to obtain 
\begin{align*}
\sum_{n\geq0}\overline{R_{4,9}}(24n+20)q^n&\equiv 9\frac{f_2^{28}f_3^3}{f_1^{13}f_4^8}+144q\frac{f_2^4f_3^3f_4^8}{f_1^5}+18\frac{f_2^2f_{6}^2f_{9}}{f_1^5} \pmod{27}\\
&\equiv 9\frac{f_2^{28}}{f_1^4f_4^8}+9qf_2^4f_1^4f_4^8-9f_2^8f_1^4 \pmod{27} \quad \text{(applying \eqref{modp^k})}.
\end{align*}

To complete the proof, it suffices to show that 
$$\frac{f_2^{28}}{f_1^4f_4^8}+qf_2^4f_1^4f_4^8-f_2^8f_1^4\equiv0 \pmod{3}.$$
Equivalently, we show $F(q):=f_2^{28}+qf_1^8f_2^4f_4^{16}-f_1^8f_2^8f_4^8\equiv0\pmod{3}$. For, we rewrite $F(q)$ as
\begin{align*}
    F(q)&=f_2^{28}+(qf_2^4f_4^{16}-f_2^8f_4^8)\frac{f_1^9}{f_1}\\
    &\equiv f_2^{28}+(qf_2^4f_4^{16}-f_2^8f_4^8)\frac{f_3^3}{f_1}\pmod{3}\quad\text{(thanks to \eqref{modp^k})}\\
    &=f_2^{28}+(qf_2^4f_4^{16}-f_2^8f_4^8)\left(\frac{f_4^3f_6^2}{f_2^2f_{12}}+q\frac{f_{12}^3}{f_4}\right) \quad \text{(using \eqref{f33/f1})}\\
    &=f_2^{28}+q\frac{f_2^2f_4^{19}f_6^2}{f_{12}}+q^2f_2^4f_4^{15}f_{12}^3-\frac{f_2^6f_4^{11}f_6^2}{f_{12}}-qf_2^8f_4^7f_{12}^3\\
    &\equiv f_2^{28}+qf_2^8f_4^{16}+q^2f_2^4f_4^{24}-f_2^{12}f_4^8-qf_2^8f_4^{16}\pmod{3} \quad\text{(thanks to \eqref{modp^k})}\\
    &=f_2^4(f_2^{24}+q^2f_4^{24}-f_2^{8}f_4^8)\\
    &=-f_2^4G(q^2),
\end{align*}
where $G(q)=f_1^8f_2^8-qf_2^{24}-f_1^{24}$ as in the proof Theorem \ref{T494n012}. As shown in the proof of the theorem, we have $G(q)\equiv0\pmod{3}$, which implies $G(q^2)\equiv0\pmod{3}$. This completes the proof.

\bibliographystyle{plain}
\nocite{*}
\bibliography{Double_Regular.bib}

\end{document}